\documentclass[12pt, a4paper]{article}
\usepackage{tikz}
\usepackage{xcolor}
\usepackage{eso-pic}

\usepackage{graphicx}
\usepackage{float}

\usepackage[colorlinks,linkcolor=red,anchorcolor=blue,citecolor=green]{hyperref}

\usepackage{amsmath}%数学公式包
\usepackage{amssymb}%特殊符号包
\usepackage{mathrsfs}%花体
\usepackage{bm}%加粗数学符号 \bm{math expression}
\usepackage{extarrows}%数学符号补充包 长等号
\usepackage{amsthm}
\usepackage{color}%Excel2LaTeX辅助包
\usepackage{booktabs}
\usepackage{multirow}
\usepackage{array}
\usepackage{comment}

\usepackage{listings}
\usepackage{bbm}

\definecolor{dkgreen}{rgb}{0,0.6,0}
\definecolor{gray}{rgb}{0.5,0.5,0.5}
\definecolor{mauve}{rgb}{0.58,0,0.82}
\usepackage[a4paper,margin=2.5cm]{geometry}
\newcommand{\Ecal}{\mathcal{E}}

\newcommand{\Pcal}{\mathcal{P}}

\newcommand{\Lcal}{\mathcal{L}}
\newcommand{\Wcal}{\mathcal{W}}

\renewcommand{\d}{\ensuremath{\mathrm{d}}}

\newcommand{\e}{ \operatorname{\mathbb E}}

\newcommand{\be}{b^{\tt{E}}}
\newcommand{\bi}{b^{\tt{I}}}
\newcommand{\Li}{L^{\tt{I}}}
\newcommand{\Real}{\mathbb{R}}
\newcommand{\ZZ}{\mathbb{Z}}
\newcommand{\abs}[1]{\left\vert#1\right\vert}
\newcommand{\norm}[1]{\left\|#1\right\|}

\newcommand{\ie}{{\it{i.e.}}}

\newcommand{\transpose}{\mathsf{T}} % or, \intercal

\newtheorem{re}{Remark}[section]
\newtheorem{lemma}{Lemma}[section]
\newtheorem{theo}{Theorem}[section]

\newtheorem{assum}{Assumption}[section]
\newtheorem{prop}{Proposition}[section]

\numberwithin{equation}{section}

\title{Empirical approximation to invariant measures of mean-field Langevin dynamics}
\author{Wenjing Cao\thanks{School of Mathematical Sciences, 
		Fudan University, 
		Shanghai 200433, China (email: {\tt wjcao22@m.fudan.edu.cn}). The author's research is supported
        by the National Key R{\&}D Program of China (No.~2022ZD0116401)} 
        \and 
        Kai Du\thanks{SCMS, Fudan University, Shanghai, China;
    Shanghai Artificial Intelligence Laboratory
    (email: {\tt kdu@fudan.edu.cn}). 
    The author's research is supported
    by the National Natural Science Foundation of China (No.~12222103),
    and by the National Key R{\&}D Program of China (No.~2022ZD0116401)
	}}
\date{\today}

\begin{document}
\maketitle
\begin{abstract}
    This paper is concerned with the approximation to invariant measures for Langevin dynamics of McKean--Vlasov type. Under dissipativity and Lipschitz conditions, we prove that the empirical measures of both the mean-field and self-interacting Langevin dynamics converge to the invariant measure in the Wasserstein distance. Numerical experiments are conducted to illustrate theoretical results.
\end{abstract}
\section{Introduction}
    In this paper, we focus on the approximation to the invariant measure of the Langevin dynamics $(X_t,V_t)$ on $\Real^{2d}$, governed by the following stochastic differential equation
    \begin{equation}\label{sde1}
        \left\{\begin{aligned}
            \d X_t&=V_t\,\d t,\\
            \d V_t&=\Big(\be(X_t)+\int_{\Real^d}\bi(X_t,x)\Lcal(X_t)(\d x)-\gamma V_t\Big)\d t+\sqrt{2\gamma}\,\d B_t.
        \end{aligned}\right.
    \end{equation}
    Here $(B_t)_{t\geq0}$ is a standard Brownian motion on $\Real^d$, and $\Lcal(X_t)$ denotes the distribution law of $X_t$. The well-posedness of \eqref{sde1} holds if coefficients $\be:\Real^d\to\Real^d$ and $\bi:\Real^d\times\Real^d\to\Real^d$ are Lipschitz continuous and the initial value of the solution has finite second moment (cf.~\cite{graham1996asymptotic}). 
    In physics, $X_t$ and $V_t$ represent the position and velocity vector of particles, respectively, governed by the external force $\be$ and the interaction force $\bi$ along with damping $\gamma$. Langevin dynamics plays a crucial role in statistical mechanics (cf.~\cite{brunger1984stochastic,schneider1978molecular}) and has broad application in machine learning (cf.~\cite{gess2024stochastic,kazeykina2020ergodicity,mei2018mean}). 
    The model of \eqref{sde1} is known as the Langevin diffusion of McKean--Vlasov type, which characterizes the mean-field limit of interacting particle systems (cf.~\cite{kac1956foundations,mckean1966class,mckean1967propagation}).
 
    The computation of invariant probability measures for McKean--Vlasov diffusions is a topic of wide interest and has been extensively explored in the literature. 
    In the context of classical Langevin dynamics with gradient-type forces, convergence rates for distribution laws $\Lcal(X_t,V_t)$ towards invariant measures have been obtained with respect to various metrics, including Sobolev norms (cf. \cite{villani2009hypocoercivity,dolbeault2009hypocoercivity,dolbeault2015hypocoercivity}), $L^2$ distances (cf.~\cite{cao2023explicit}), and Wasserstein metrics (cf.~\cite{10.1214/18-AOP1299}). 
    In mean-field settings, where drift terms are dependent on distribution, contraction and propagation of chaos are studied for equations with Brownian noise (cf.~\cite{schuh2022global,kazeykina2020ergodicity,guillin2021kinetic}) or general L{\'e}vy noise (cf.~\cite{liu2024exponential}). 
    To the best of our knowledge, no existing works so far have delved into the empirical approximation to invariant measures for Langevin dynamics of McKean--Vlasov type with general external and interaction forces. This paper is aimed to fill the gap.

    Our work builds upon the basis of \cite{du2023empirical}, where an efficient and space-saving algorithm was put forward to approximate the invariant measure of first-order McKean--Vlasov dynamics. The authors argued that empirical measures, namely, $\Ecal_t(X)=\frac{1}{t}\int_0^t\delta_{X_s}\,\d s$ for a stochastic process $X$, of both the McKean--Vlasov process and its associated path-dependent process can converge to the invariant measure. 
    However, results in \cite{du2023empirical} necessitate a strong monotonicity condition (cf.~\cite[Assumption 1.1]{du2023empirical}), rendering them inapplicable to such second-order equations as \eqref{sde1}. On the other hand, it is proved in \cite{schuh2022global} that under specific Lipschitz and dissipativity conditions, there exists a unique invariant measure for Langevin dynamics of McKean--Vlasov type. 
    Motivated by these findings, we are interested in the empirical ergodicity of mean-field Langevin diffusions, under the same assumptions as those in \cite{schuh2022global}. 
    Our objective is to show that empirical ergodicity is an inherent characteristic of underdamped Langevin dynamics with Lipschitz continuous and dissipative coefficients, without the need for any supplementary constraints.    
    
    In particular, we strive to approximate the invariant measure $\mu^*$ of \eqref{sde1} by empirical measures
    \[\Ecal_t(X,V):=\int_{0}^{1}\delta_{(X_{ts},V_{ts})}\,\d s=\frac{1}{t}\int_0^t\delta_{(X_s,V_s)}\,\d s\]
    of the mean-field Langevin dynamics and its self-interacting counterpart
    \begin{equation}\label{sde3}
        \left\{\begin{aligned}
            \d Z_t&=W_t\,\d t,\\
            \d W_t&=\Big(\be(Z_t)+\int_0^1\bi(Z_t,Z_{ts})\d s-\gamma W_t\Big)\d t+\sqrt{2\gamma}\,\d B_t,
        \end{aligned}\right.
    \end{equation}
    where the distribution dependence on $\Lcal(X_t)$ is replaced by the path dependence on $\Ecal_t(Z)$. 
    By imposing Lipschitz continuity and dissipativity on $\be$ and $\bi$ (see Assumption \ref{dissi}), we obtain the convergence rates of empirical measures characterized by upper bound estimates for the $1$-Wasserstein distance (see Theorems \ref{MeanFieldConvergence} and \ref{SelfInteractingConvergence}). 
    Note that such dynamics as \eqref{sde3} are also referred to as self-interacting diffusions, which are traced back to~\cite{benaim2002self,cranston1995self,raimond1997self,10.1214/EJP.v17-2121}. 
    Some innovative investigation into the application of self-interacting diffusions can be found in a recent paper~\cite{du2023self}. 

    Following the ideas in \cite{cao2023empirical}, we compare \eqref{sde1} and~\eqref{sde3} with the Markovian equation derived from them:
    \begin{equation}\label{sde2}
        \left\{\begin{aligned}
            \d Y_t&=U_t\,\d t,\\
            \d U_t&=\Big(\be(Y_t)+\int_{\Real^d}\bi(Y_t,x)\mu_X^*(\d x)-\gamma U_t\Big)\d t+\sqrt{2\gamma}\,\d B_t,
        \end{aligned}\right.
    \end{equation}
    where $\mu_X^*$ denotes the marginal distribution of $\mu^*$ on $\Real^d$, \ie~$\mu_X^*(\cdot)=\mu^*(\cdot\times\Real^d)$. 
    Empirical ergodicity of \eqref{sde2} is guaranteed by its exponential contractivity (see Theorem~\ref{MarkovConvergence}). In contrast to first-order scenarios in~\cite{cao2023empirical,du2023empirical}, the degenerate nature of Langevin diffusions adds to complexity and challenges in research. 
    Hence, we adopt the reflection coupling technique (cf.~\cite{eberle2016reflection}) as well as the construction for metrics in~\cite{10.1214/18-AOP1299,schuh2022global} to estimate the error.       

    The remainder of our paper is arranged as follows. 
    Section~\ref{results} states the main results. Proofs of Theorems~\ref{MeanFieldConvergence} and \ref{SelfInteractingConvergence} are given in Section~\ref{proofs}. 
    Section~\ref{appendix} involves the proof of Theorem~\ref{MarkovConvergence}, which lays a solid foundation for main theorems in Section~\ref{results}. 
    In Section~\ref{experiment}, numerical results are demonstrated.

    We end the introduction part by clarifying notations. The Euclidean norm and the inner product in $\Real^d$ are denoted by $\abs{\cdot}$ and $\langle\cdot,\cdot\rangle$ respectively. For any mapping $f:\Real^n\to\Real^m$, $\abs{f}_1$ denotes its Lipschitz constant, \ie $$\abs{f}_1:=\sup\limits_{x\neq y}\frac{\abs{f(x)-f(y)}}{\abs{x-y}}.$$
    Let $\Pcal(\Real^d)$ be the space of all Borel probability measures on $\Real^d$, and the $1$-Wasserstein distance between any $\mu,\,\nu\in\Pcal(\Real^d)$ is defined as 
	 $$
	 \Wcal_1(\mu,\nu):=\inf\big\{\e[\abs{X-Y}]:\, \Lcal(X)=\mu,\, \Lcal(Y)=\nu\big\}.$$
    For any $\mu\in\Pcal(\Real^d)$, we define
    \[\norm{\mu}_1:=\int_{\Real^d}\abs{x}\mu(\d x)\]
    and
    \[\Pcal_1(\Real^d):=\big\{\mu\in\Pcal(\Real^d):\, \norm{\mu}_1<\infty\big\}.\] 
    Note that $(\Pcal_1(\Real^d),\Wcal_1)$ is a complete metric space (cf. \cite{villani2009optimal}).

\section{Main results}\label{results}
    The following assumption implies the Lipschitz property as well as dissipativity of $\be(\cdot)$ and $\bi(\cdot)$, which coincides with \cite[Assumptions 2, 3]{schuh2022global} to establish the global contractivity for the Langevin dynamics there.
    \begin{assum}\label{dissi}
        For the coefficients $\be:\Real^d\to\Real^d$ and $\bi:\Real^{2d}\to\Real^d$, we assume that
        
        \emph{(i)} $\vert\bi\vert_1=:\Li <\infty $;
        
        \emph{(ii)} there exists a positive definitive symmetric matrix $K\in\Real^{d\times d}$ with the smallest eigenvalue $\kappa$ and the largest eigenvalue $L_K$ such that
        \[\be(x)=-Kx+g(x),\,\forall x\in\Real^d,\]
        where $g:\Real^d\to\Real^d$ is a Lipschitz function with $\abs{g}_1=:L_g<\infty$, and satisfies 
        \[\langle g(x)-g(x^{\prime}),x-x^{\prime}\rangle\leq 0,\,\forall x,x^{\prime}\in\Real^d\ \text{\rm{such that}}\ \abs{x-x^{\prime}}\geq R\]
        with some constant $R\in[0,\infty)$.
    \end{assum}

    The main results of this paper consist of the following two theorems. 
    For the mean-field Langevin dynamics~\eqref{sde1}, we obtain the empirical ergodicity property.
    \begin{theo}\label{MeanFieldConvergence}
        Suppose Assumption~\ref{dissi} holds, with $L_g$ and $\Li$ satisfying 
        \begin{equation}\label{ineq01}
            L_g<\gamma\sqrt{{\kappa}/{2}}.
        \end{equation} 
        and
        \begin{equation}\label{ineq02}
            \Li\leq \mathbf{C}_0\min\Big\{\frac{\gamma\tau}{12}\sqrt{2\kappa}\min\{1,2(L_K+L_g)\gamma^{-2}\},\frac{L_K+L_g}{4}\Big\}.
        \end{equation}
        Here
        \begin{equation}\label{tau}
            \tau:=\min\{1/8,\gamma^{-2}\kappa/2-\gamma^{-4}L_g^2\},
        \end{equation}
        and
        \begin{equation}\label{C0}
            \mathbf{C}_0:=\exp\Big(-\frac{L_K+L_g}{4}R_1^2\Big)\Bigg(1+\Big(1-\exp\Big(-\frac{L_K+L_g}{4}R_1^2\Big)\Big)\kappa(L_K+L_g)^{-1}\Bigg)^{-1}
        \end{equation}
        with $R_1$ given by~\eqref{R1} in Section~\ref{preliminaries}. Let $\mu^*$ be the invariant probability measure of~\eqref{sde1}.
        Then for any solution $(X_t,V_t)$ of~\eqref{sde1} whose initial value satisfies   
        \[\e[\abs{X_0}^2+\abs{V_0}^2]<\infty,\]
        there exists a positive constant $C$, independent of $t$, such that
        \[\e[W_1(\Ecal_t(X,V),\mu^*)]\leq C t^{-\varepsilon}\]
        for any $0<\varepsilon<\min\{\frac{1}{4},\frac{1}{2d}\}$.
    \end{theo}     
    Under the same setting, we prove that the empirical measure of the self-interacting Langevin dynamics~\eqref{sde3} also converges to the invariant measure $\mu^*$ for~\eqref{sde1}.
    \begin{theo}\label{SelfInteractingConvergence}
        Suppose Assumption~\ref{dissi}, \eqref{ineq01} and~\eqref{ineq02} hold. Then for any solution $(Z_t,W_t)$ of~\eqref{sde3} whose initial value satisfies 
        \[\e[\abs{Z_0}^2+\abs{W_0}^2]<\infty,\]
        there exists a positive constant $C$, independent of $t$, such that
        \[\e[W_1(\Ecal_t(Z,W),\mu^*)]\leq C t^{-\varepsilon}\]
        for any $0<\varepsilon<\min\{\frac{1}{4},\frac{1}{2d},\frac{\mathbf{C}_3}{\mathbf{C}_3+\gamma^{-1}\Li/2}\}$, where the value of $\mathbf{C}_3=\mathbf{C}_3(L_K,L_g,\kappa,\gamma)>0$ is given by~\eqref{C3} in Section \ref{SelfInteracting}. 
    \end{theo}
    The theorems are proved in Section~\ref{proofs}.
    Note that a probability measure $\mu^*$ is called invariant for \eqref{sde1} if $P_t^* \mu^*=\mu^*$ for all $t\geq0$, where $(P_t^*)_{t\geq0}$ denotes the nonlinear transition semigroup induced by \eqref{sde1}, \ie~$P_t^*\Lcal(X_0,V_0)=\Lcal(X_t,V_t)$.
    The existence and uniqueness of $\mu^*$ are ensured by~\cite[Theorem 12]{schuh2022global} under the same conditions of the above theorems.
    
    \begin{re}
        Assumption~\ref{dissi} suggests that $\be$ is a globally Lipschitz continuous function such that $\abs{b}_1\leq L_K+L_g$ and that
        \[\langle \be(x)-\be(x^{\prime}),x-x^{\prime}\rangle\leq -\kappa\abs{x-x^{\prime}}^2\]
        for all $\abs{x-x^{\prime}}\geq R$. It also implies the weak interaction condition as~\cite[Assumption 1.1]{du2023sequential}.
    \end{re}
\section{Proofs of Theorems \ref{MeanFieldConvergence} and \ref{SelfInteractingConvergence}}\label{proofs}
\subsection{Preliminaries}\label{preliminaries}  
    As a key step, we acquire the following estimate for \eqref{sde2}, whose proof is postponed to Section~\ref{appendix}. 
    \begin{theo}\label{MarkovConvergence}
        Suppose Assumption~\ref{dissi}, \eqref{ineq01} and \eqref{ineq02} hold. Then for any solution $(Y_t,U_t)$ of~\eqref{sde2} whose initial value satisfies 
        \[\e[\abs{Y_0}^2+\abs{U_0}^2]<\infty,\]
        there exists a positive constant $C$, independent of $t$, such that
        \[\e[W_1(\Ecal_t(Y,U),\mu^*)]\leq C t^{-\varepsilon}\]
        for any $0<\varepsilon<\min\{\frac{1}{4},\frac{1}{2d}\}$.
    \end{theo}
    
    We also introduce some auxiliary constants and functions. 
    Firstly, we define norms $r_l,r_s:\Real^d\times\Real^d\to [0,\infty)$ such that
    \begin{equation}\label{large_small_norms}
        \begin{aligned}
        r_l(x,v)^2&:=\gamma^{-2}\langle Kx,x\rangle+\frac{1}{2}\gamma^{-2}\abs{v}^2+\frac{1}{2}\abs{(1-2\tau)x+\gamma^{-1}v}^2,\\
        r_s(x,v)&:=\alpha\abs{x}+\abs{x+\gamma^{-1}v},
        \end{aligned}
    \end{equation}
    where $\tau$ is given by~\eqref{tau} and~$\alpha:=2(L_K+L_g)\gamma^{-2}$.
    Then we have (cf.~\cite{schuh2022global})
    \begin{equation}
        2\epsilon_0 r_l\leq r_s\leq \frac{1}{E_0}r_l,
    \end{equation}
    where
    \begin{equation}\label{epsilon}
    \begin{aligned}
        \epsilon_0&:=\frac{1}{2}\min\big\{1,\frac{2\alpha\gamma}{3\sqrt{L_K}},\alpha\big\},\\
        E_0&:=\frac{1}{2}\min\big\{1,\frac{\sqrt{\kappa}}{\sqrt{2}\alpha\gamma}\big\}. 
    \end{aligned}
    \end{equation} 
    Secondly, we define
    \begin{equation}\label{Delta}
        \Delta(x,v):=r_s(x,v)-\epsilon_0 r_l(x,v)
    \end{equation}
    and
    \begin{equation}\label{dis_compact}
        D(\tilde{R}):=\sup\limits_{r_l(x,v)\leq \Tilde{R}}\Delta(x,v),
    \end{equation}
    where
    \begin{equation}\label{tildeR}
        \Tilde{R}:=\frac{1}{\tau\gamma^2}(8\cdot\mathbb{I}_{\{R>0\}}+L_g R^2).
    \end{equation}
    Based on \eqref{dis_compact} and \eqref{tildeR}, we introduce the constant
    \begin{equation}\label{R1}
        R_1:=\sup\limits_{\Delta(x,v)\leq D(\tilde{R})}r_s(x,v)
    \end{equation}
    and an auxiliary function $f\in C^2[0,\infty)$ adapted from~\cite{du2023sequential,liu2021long} such that
    \begin{equation}\label{auxiliary}
	f^{\prime}(r)=\frac{1}{2}\int_{r}^{\infty}s \exp\Big(-\frac{1}{2}\int_{r}^{s}x\theta(x)\,\d x\Big)\d s,\quad f(0)=0,
    \end{equation}
    for all $r\geq0$, where
    \begin{equation}
        \theta(r):=\begin{cases}
            -(L_K+L_g), &r<R_1\\
            \kappa, &r\geq R_1.
        \end{cases}
    \end{equation}
    Note that $f$ is a concave function satisfying
    \begin{equation}\label{ODE}
	2f^{\prime\prime}(r)-r\theta(r)f^{\prime}(r)=-r
    \end{equation}
    and
    \begin{equation}\label{auxiliary_Lipschitz}
    \begin{aligned}
        f(r)/r\geq\, &f^{\prime}(R_1)=\kappa^{-1},\\
        f(r)/r\leq\, &f^{\prime}(0)=\exp\Big(\frac{L_K+L_g}{4}R_1^2\Big)\kappa^{-1}+\bigg(\exp\Big(\frac{L_K+L_g}{4}R_1^2\Big)-1\bigg)(L_K+L_g)^{-1}
    \end{aligned}
    \end{equation}
    for all $r\geq0$ (cf.~\cite[Lemma 5.1]{du2023sequential}). Note that $f^{\prime}(R_1)f^{\prime}(0)^{-1}$ is equal to $\mathbf{C}_0$ in \eqref{C0}.
    \begin{re}
        If $R=0$, we have $\Tilde{R}=R_1=D(\Tilde{R})=0$, and thus $f(r)=\frac{1}{\kappa}r$ for all $r\geq0$.
    \end{re}

\subsection{Convergence of mean-field dynamics}\label{meanfield}
    In this section, we prove Theorem \ref{MeanFieldConvergence}.
    Let $\mathbf{\lambda},\,\mathbf{\pi}:\Real^d\times\Real^d\to [0,1]$ be smooth functions satisfying $\mathbf{\lambda}(\cdot)^2+\mathbf{\pi}(\cdot)^2=1$ and 
    \begin{equation}\label{lambda}
        \lambda(x,v)=\begin{cases}
		1, &\abs{x+\gamma^{-1}v}\geq\delta\ \text{and}\ \Delta(x,v)\leq D(\tilde{R})\ \text{and}\  D(\tilde{R})>0,\\
		0, &\abs{x+\gamma^{-1}v}=0\ \text{or}\ \Delta(x,v)\geq D(\tilde{R})+\delta\cdot\mathbb{I}_{\{D(\tilde{R})>0\}}.
        \end{cases}
    \end{equation}
    Here the parameter $0<\delta<1$ will be optimized later. Let $(Y_t,U_t)$ be the solution of \eqref{sde2} such that $(Y_0,U_0)=(X_0,V_0)$, and $(\hat{B}_t)$ be a standard Brownian motion on $\Real^d$ independent of $(B_t)$ and $(X_0.V_0)$. We consider the following reflection coupling
    \begin{align}
    \label{sde1_coupling}
        &\left\{\begin{aligned}
            \d \hat{X}_t&=\hat{V}_t\,\d t,\\
            \d \hat{V}_t&=\Big(\be(\hat{X}_t)+\int_{\Real^d}\bi(\hat{X}_t,x)\Lcal(\hat{X}_t)(\d x)-\gamma \hat{V}_t\Big)\d t\\
            &\quad+\sqrt{2\gamma}\mathbf{\lambda}(F_t,G_t)\,\d B_t+\sqrt{2\gamma}\mathbf{\pi}(F_t,G_t)\,\d \hat{B}_t,
        \end{aligned}\right.\\
    \label{sde2_coupling}
        &\left\{\begin{aligned}
            \d \hat{Y}_t&=\hat{U}_t\,\d t,\\
            \d \hat{U}_t&=\Big(\be(\hat{Y}_t)+\int_{\Real^d}\bi(\hat{Y}_t,x)\mu_X^*(\d x)-\gamma \hat{U}_t\Big)\d t\\
            &\quad +\sqrt{2\gamma}\mathbf{\lambda}(F_t,G_t)(I_d-2e_t e_t^{\transpose})\,\d B_t+\sqrt{2\gamma}\mathbf{\pi}(F_t,G_t)\,\d \hat{B}_t,
        \end{aligned}\right.
    \end{align}
    such that 
    $$(\hat{X}_0,\hat{V}_0)=(X_0,V_0),\quad (\hat{Y}_0,\hat{U}_0)=(Y_0,U_0).$$
    Here 
    \[
        F_t:=\hat{X}_t-\hat{Y}_t,\quad G_t:=\hat{V}_t-\hat{U}_t,
    \]
    and 
    \[
        e_t:=\begin{cases}
            \frac{F_t+\gamma^{-1}G_t}{\abs{F_t+\gamma^{-1}G_t}}, &F_t+\gamma^{-1}G_t\neq 0,\\
            0, &F_t+\gamma^{-1}G_t=0.
        \end{cases}
    \]
    Then we have, for all $t\geq0$,
    $$\Lcal(\hat{X}_t,\hat{V}_t)=\Lcal(X_t,V_t),\quad\Lcal(\hat{Y}_t,\hat{U}_t)=\Lcal(Y_t,U_t).$$  
    Let $H_t:=F_t+\gamma^{-1}G_t$, and we get
    \begin{equation}\label{distance01}
        \begin{aligned}
            \d F_t&=G_t\,\d t=\gamma(H_t-F_t)\,\d t,\\
            \d G_t&=-\gamma G_t\,\d t+\Big(\be(\hat{X}_t)-\be(\hat{Y}_t)+\int_{\Real^d}\bi(\hat{X}_t,x)\Lcal(\hat{X}_t)(\d x)-\int_{\Real^d}\bi(\hat{Y}_t,x)\mu_X^*(\d x)\Big)\d t\\
            &\quad+2\sqrt{2\gamma}e_t e_t^{\transpose}\mathbf{\lambda}(F_t,G_t)\,\d B_t,\\
            \d H_t&=\gamma^{-1}\Big(\be(\hat{X}_t)-\be(\hat{Y}_t)+\int_{\Real^d}\bi(\hat{X}_t,x)\Lcal(\hat{X}_t)(\d x)-\int_{\Real^d}\bi(\hat{Y}_t,x)\mu_X^*(\d x)\Big)\d t\\
            &\quad+2\sqrt{2\gamma^{-1}}e_t e_t^{\transpose}\mathbf{\lambda}(F_t,G_t)\,\d B_t.\\
        \end{aligned}
    \end{equation}
    In this section, we abbreviate 
    \begin{equation}\label{abbreviation}
    r_l(t):=r_l(F_t,G_t),\quad r_s(t):=r_s(F_t,G_t),\quad\Delta(t):=\Delta(F_t,G_t)
    \end{equation}
    with $r_s,\,r_l,\,\Delta:\Real^{2d}\to[0,\infty)$ defined in Section \ref{preliminaries}. 

    To prove convergence, we require the following two lemmas, demonstrating convergence in large and small distances, respectively.
    \begin{lemma}\label{lemma_case1}
        For $t\geq0$ such that $\Delta(t)\geq D(\tilde{R})$, we have
        \[\begin{aligned}
        \d r_l(t)
        &\leq \Big(-c_1 r_l(t)+\gamma^{-1}\Li\big(\Wcal_1(\Lcal(X_t),\mu_X^*)+\abs{F_t}\big)\Big)\d t\\
        &\quad+r_l(t)^{-1}\sqrt{2\gamma^{-1}}\mathbf{\lambda}(F_t,G_t)\big((1-2\tau)\gamma^{-1}F_t+2\gamma^{-2}G_t\big)^{\mathsf{T}}e_t e_t^{\transpose}\d B_t,
    \end{aligned}\]
    where $c_1=\frac{\tau\gamma}{2}$.
    \end{lemma}
    \begin{proof}
    Let
    \[A:=\gamma^{-2}K+\frac{1}{2}(1-2\tau)^2I_d,\quad B:=(1-2\tau)\gamma^{-1}I_d,\quad C:=\gamma^{-2}I_d,\]
    and we have
    \[r_l(x,v)^2=\langle Ax,x\rangle+\langle Bx,v\rangle+\langle Cv,v\rangle.\]
    By Ito's formula we have
    \[
    \begin{aligned}
        \d r_l(t)^2 =\,&\d\big(\langle AF_t,F_t\rangle+\langle BF_t,G_t\rangle+\langle CG_t,G_t\rangle\big)\\
        \leq\,& 2\langle AF_t,G_t\rangle\,\d t+\big(\langle BG_t,G_t\rangle-\gamma\langle BF_t,G_t\rangle-\langle BF_t,KF_t\rangle \big)\,\d t\\
        &+L_g(1-2\tau)\gamma^{-1}\abs{F_t}^2\cdot\mathbb{I}_{\{\abs{F_t}<R\}}\,\d t\\
        &+\big(-2\gamma\langle CG_t,G_t\rangle-2\langle KF_t,CG_t\rangle+2\gamma^{-2}L_g \abs{F_t}\abs{G_t}\big)\,\d t\\  &+\abs{BF_t+2CG_t}\Li\big(\Wcal_1(\Lcal(X_t),\mu_X^*)+\abs{F_t}\big)\,\d t\\
        &+8\gamma^{-1}\mathbf{\lambda}(F_t,G_t)^2\d t+2\sqrt{2\gamma}\mathbf{\lambda}(F_t,G_t)(BF_t+2CG_t)^{\transpose}e_t e_t^{\transpose}\d B_t\\
        \leq\,& \big\langle (-BK+\gamma^{-3}L_g^2)F_t,F_t\big\rangle\,\d t+\big((1-2\tau)\gamma^{-1}-\gamma^{-1}\big)\abs{G_t}^2\d t\\
        &+\langle (2A-\gamma B-2KC)F_t,G_t\rangle\,\d t+L_g(1-2\tau)\gamma^{-1}\abs{F_t}^2\cdot\mathbb{I}_{\{\abs{F_t}<R\}}\,\d t\\
        &+\abs{(1-2\tau)\gamma^{-1}F_t+2\gamma^{-2}G_t}\Li\big(\Wcal_1(\Lcal(X_t),\mu_X^*)+\abs{F_t}\big)\,\d t\\
        &+8\gamma^{-1}\mathbf{\lambda}(F_t,G_t)^2\d t+2\sqrt{2\gamma}\mathbf{\lambda}(F_t,G_t)\big((1-2\tau)\gamma^{-1}F_t+2\gamma^{-2}G_t\big)^{\transpose}e_t e_t^{\transpose}\d B_t.
    \end{aligned}
    \]
    Here we use
    \[2\gamma^{-2}L_g\abs{F_t}\abs{G_t}=\gamma^{-1}(2\gamma^{-1}L_g\abs{F_t}\cdot\abs{G_t})\leq \gamma^{-1}\big(\gamma^{-2}L_g^2\abs{F_t}^2+\abs{G_t}^2\big).\]
    Note that 
    \[\begin{aligned}
        \langle (-BK+\gamma^{-3}L_g^2)F_t,F_t\rangle&=-(1-2\tau)\gamma^{-1}\langle KF_t,F_t\rangle+\gamma^{-3}L_g^2\abs{F_t}^2\\
        &\leq -2\tau\gamma\langle\gamma^{-2}KF_t,F_t\rangle-\gamma^{-1}((1-4\tau)\kappa-L_g^2\gamma^{-2})\abs{F_t}^2\\
        &\leq -2\tau\gamma\langle\gamma^{-2}KF_t,F_t\rangle+(-\frac{1}{2}\kappa\gamma^{-1}+L_g^2\gamma^{-3})\abs{F_t}^2\\
        &\leq -2\tau\gamma\langle\gamma^{-2}KF_t,F_t\rangle-\tau\gamma\abs{F_t}^2\\
        &\leq -2\tau\gamma\big(\langle\gamma^{-2}KF_t,F_t\rangle-\frac{1}{2}(1-2\tau)^2\abs{F_t}^2\big).\\
    \end{aligned}\]
    Thus we get
    \[
    \begin{aligned}
        \d r_l(t)^2
        \leq\,& -2\tau\gamma r_l(t)^2\,\d t+L_g(1-2\tau)\gamma^{-1}\abs{F_t}^2\cdot\mathbb{I}_{\{\abs{F_t}<R\}}\,\d t\\
        &+\gamma^{-1}\abs{(1-2\tau)F_t+2\gamma^{-1}G_t}\Li\big(\Wcal_1(\Lcal(X_t),\mu_X^*)+\abs{F_t}\big)\,\d t\\
        &+8\gamma^{-1}\mathbf{\lambda}(F_t,G_t)^2\d t+2\sqrt{2\gamma^{-1}}\mathbf{\lambda}(F_t,G_t)\big((1-2\tau)\gamma^{-1}F_t+2\gamma^{-2}G_t\big)^{\transpose}e_t e_t^{\transpose}\d B_t.
    \end{aligned}
    \]
    Since $\Delta(t)\geq D(\tilde{R})$, by \eqref{dis_compact} and \eqref{lambda} we have $r_l(t)^2\geq \Tilde{R}$ and $\mathbf{\lambda}(F_t,G_t)\leq \mathbb{I}_{\{R>0\}}$. Thus by \eqref{tildeR} we get
    \[\begin{aligned}
        &-\tau\gamma r_l(t)^2+L_g(1-2\tau)\gamma^{-1}\abs{F_t}^2\cdot\mathbb{I}_{\{\abs{F_t}<R\}}+8\gamma^{-1}\mathbf{\lambda}(F_t,G_t)^2\\
        \leq& -\tau\gamma\tilde{R}+L_g R^2\gamma^{-1}+8\gamma^{-1}\mathbb{I}_{\{R>0\}}\leq 0.
    \end{aligned}\]
    Applying Ito's formula again we have
    \[\begin{aligned}
        \d &r_l(t)\leq \frac{1}{2r_l(t)}\d r_l(t)^2\\
        \leq& -\frac{\tau\gamma}{2} r_l(t)\,\d t+(2\gamma r_l(t))^{-1}\abs{(1-2\tau)F_t+2\gamma^{-1}G_t}\Li\big(\Wcal_1(\Lcal(X_t),\mu_X^*)+\abs{F_t}\big)\,\d t\\
        &+r_l(t)^{-1}\sqrt{2\gamma^{-1}}\mathbf{\lambda}(F_t,G_t)\big((1-2\tau)\gamma^{-1}F_t+2\gamma^{-2}G_t\big)^{\transpose}e_t e_t^{\transpose}\d B_t\\
        \leq& \Big(-\frac{\tau\gamma}{2} r_l(t)+\gamma^{-1}\Li\big(\Wcal_1(\Lcal(X_t),\mu_X^*)+\abs{F_t}\big)\Big)\d t\\
        &+r_l(t)^{-1}\sqrt{2\gamma^{-1}}\mathbf{\lambda}(F_t,G_t)\big((1-2\tau)\gamma^{-1}F_t+2\gamma^{-2}G_t\big)^{\transpose}e_t e_t^{\transpose}\d B_t,
    \end{aligned}\]
    where we use 
    \[(2r_l(t))^{-1}\abs{(1-2\tau)F_t+2\gamma^{-1}G_t}=\frac{1}{2}\sqrt{\frac{\abs{(1-2\tau)F_t+2\gamma^{-1}G_t}^2}{r_l(t)^2}}\leq 1.\]
    The proof is complete.
    \end{proof}
    \begin{lemma}\label{lemma_case2}
        For $t\geq0$ such that $\Delta(t)< D(\tilde{R})$, we have
        \[\begin{aligned}
        \d f(r_s(t))
        &\leq -c_2f(r_s(t))\d t+\gamma^{-1}\Li\big(\Wcal_1(\Lcal(X_t),\mu_X^*)+\abs{F_t}\big)\d t\\
        &\quad-\frac{\gamma\alpha}{4}f^{\prime}(R_1)\abs{F_t}\d t+(\alpha+\frac{1}{2})\delta\gamma\,\d t +\d \mathbf{M}_t,
        \end{aligned}\]
        where $f$ is given by \eqref{auxiliary}, $(\mathbf{M}_t)$ is a martingale, and 
        \[c_2=f^{\prime}(0)^{-1}\min\Big\{2\gamma^{-1},\frac{\gamma}{4}f^{\prime}(R_1)\Big\}\in(0,\infty).\]
    \end{lemma}
    \begin{proof}
        Note that $\d F_t=\gamma(H_t-F_t)\d t$, and then by Ito--Tanaka formula we have
        \[
        \begin{aligned}
            \d\abs{F_t}\left\{
            \begin{aligned}
                &=\big\langle\frac{F_t}{\abs{F_t}},\gamma(H_t-F_t)\big\rangle\d t\leq (-\gamma\abs{F_t}+\gamma\abs{H_t})\d t,\ &F_t\neq 0\\
                &\leq \gamma\abs{H_t}\d t,\ &F_t=0.
            \end{aligned}
            \right.
        \end{aligned}
        \]
        Thus we get 
    \[\d \abs{F_t}\leq (-\gamma\abs{F_t}+\gamma\abs{H_t})\,\d t.\]
    Applying Ito--Tanaka formula to $\abs{H_t}$, we get by \eqref{distance01}, and then by Assumption \ref{dissi}:
    \[\begin{aligned}
        \d\abs{H_t}=\,&\gamma^{-1}\Big\langle e_t,\be(\hat{X}_t)-\be(\hat{Y}_t)+\int_{\Real^d}\bi(\hat{X}_t,x)\Lcal(\hat{X}_t)(\d x)-\int_{\Real^d}\bi(\hat{Y}_t,x)\mu_X^*(\d x)\Big\rangle\d t\\
        &+2\sqrt{2\gamma^{-1}}\mathbf{\lambda}(F_t,G_t)e_t^{\transpose}\d B_t,\\
        \leq& \gamma^{-1}(L_K+L_g)\abs{F_t}\d t+\gamma^{-1}\Li\big(\abs{F_t}+\Wcal_1(\Lcal(X_t),\mu_X^*)\big)\,\d t+2\sqrt{2\gamma^{-1}}\mathbf{\lambda}(F_t,G_t)e_t^{\transpose}\d B_t.
    \end{aligned}\]
    Hence, by $r_s(t)=\alpha\abs{F_t}+\abs{H_t}$ we have 
    \[
    \begin{aligned}
        \d r_s(t)\leq& \Big(\gamma^{-1}(L_K+L_g-\alpha\gamma^2)\abs{F_t}+\alpha\gamma\abs{H_t}+\gamma^{-1}\Li\big(\abs{F_t}+\Wcal_1(\Lcal(X_t),\mu_X^*)\big)\Big)\d t\\
        &+2\sqrt{2\gamma^{-1}}\mathbf{\lambda}(F_t,G_t)e_t^{\transpose}\d B_t\\
        =&\Big(-\frac{1}{2}\alpha\gamma\abs{F_t}+\alpha\gamma\abs{H_t}+\gamma^{-1}\Li\big(\abs{F_t}+\Wcal_1(\Lcal(X_t),\mu_X^*)\big)\Big)\d t\\
        &+2\sqrt{2\gamma^{-1}}\mathbf{\lambda}(F_t,G_t)e_t^{\transpose}\d B_t
    \end{aligned}\]
    By Ito's formula, we have
    \[
    \begin{aligned}
        \d f(r_s(t))
        \leq\, &f^{\prime}(r_s(t))\Big(-\frac{1}{2}\alpha\gamma\abs{F_t}+\alpha\gamma\abs{H_t}+\gamma^{-1}\Li\big(\abs{F_t}+\Wcal_1(\Lcal(X_t),\mu_X^*)\big)\Big)\d t\\
        &+2f^{\prime}(r_s(t))\sqrt{2\gamma^{-1}}\mathbf{\lambda}(F_t,G_t)e_t^{\transpose}\d B_t+4f^{\prime\prime}(r_s(t))\gamma^{-1}\mathbf{\lambda}(F_t,G_t)^2\d t.
    \end{aligned}\]

    If $\abs{H_t}>\delta$, by $\Delta(t)< D(\tilde{R})$ we have $\mathbf{\lambda}(F_t,G_t)=1,\,r_s(t)<R_1$. By~\eqref{ODE},~\eqref{auxiliary_Lipschitz}, and the concavity of $f$, we obtain
    \[\begin{aligned}
        \d f(r_s(t))
        \leq\,&\big(\alpha\gamma f^{\prime}(r_s(t))r_s(t)+4f^{\prime\prime}(r_s(t))\gamma^{-1}\big)\d t-\frac{1}{2}\alpha\gamma f^{\prime}(r_s(t))\abs{F_t}\d t\\
        &+\gamma^{-1}\Li f^{\prime}(0)\big(\abs{F_t}+\Wcal_1(\Lcal(X_t),\mu_X^*)\big)\d t+\d \mathbf{M}_t\\
        =&-2\gamma^{-1}r_s(t)\d t-\frac{1}{2}\alpha\gamma f^{\prime}(r_s(t))\abs{F_t}\d t\\
        &+\gamma^{-1}\Li f^{\prime}(0)\big(\abs{F_t}+\Wcal_1(\Lcal(X_t),\mu_X^*)\big)\d t+\d \mathbf{M}_t\\
        \leq& -2\gamma^{-1}f^{\prime}(0)^{-1}f(r_s(t))\d t+\gamma^{-1}\Li f^{\prime}(0)\big(\abs{F_t}+\Wcal_1(\Lcal(X_t),\mu_X^*)\big)\d t\\
        &-\frac{\gamma\alpha}{2}f^{\prime}(R_1)\abs{F_t}\d t+\d \mathbf{M}_t.
    \end{aligned}\]

    If $\abs{H_t}\leq\delta$, we have 
    \[-\frac{1}{2}\alpha\abs{F_t}+\alpha\abs{H_t}\leq -\frac{1}{2}(r_s(t)-\delta)+\alpha\delta=-\frac{1}{2}r_s(t)+(\alpha+\frac{1}{2})\delta.\]
    Hence, by~\eqref{ODE},~\eqref{auxiliary_Lipschitz}, and concavity we get
    \[
    \begin{aligned}
        \d f(r_s(t))\leq& -\frac{\gamma}{2}f^{\prime}(r_s(t))r_s(t)\d t+(\alpha+\frac{1}{2})\gamma\delta\,\d t\\
        &+\gamma^{-1}\Li f^{\prime}(0)\big(\abs{F_t}+\Wcal_1(\Lcal(X_t),\mu_X^*)\big)\d t+\d \mathbf{M}_t\\
        \leq& -\frac{\gamma}{4}f^{\prime}(R_1)f^{\prime}(0)^{-1}f(r_s(t))\d t-\frac{\gamma\alpha}{4}f^{\prime}(R_1)\abs{F_t}\d t+(\alpha+\frac{1}{2})\gamma\delta\,\d t\\
        &+\gamma^{-1}\Li f^{\prime}(0)\big(\abs{F_t}+\Wcal_1(\Lcal(X_t),\mu_X^*)\big)\d t+\d \mathbf{M}_t.\\
    \end{aligned}
    \]

    Combining the above two cases, we finish the proof.
    \end{proof}

    Let us turn back to the proof of Theorem~\ref{MeanFieldConvergence}.
    We introduce the function
    \begin{equation}\label{rho}
        \rho(x,v):=f\big(\Delta(x,v)\wedge D(\tilde{R})+\epsilon_0 r_l(x,v)\big).
    \end{equation}
    By \cite[Lemma 16]{schuh2022global}, it holds for all $(x,v)\in\Real^d\times\Real^d$ that
    \begin{equation}\label{metric_equi}
        \rho(x,v)\geq \mathbf{C}_1\abs{(x,v)},
    \end{equation}
    with the  constant 
    \begin{equation}\label{C1}
            \mathbf{C}_1=f^{\prime}(R_1)\epsilon_0\gamma^{-1}\min\{\sqrt{\kappa},\frac{\sqrt{2}}{2}\}.
    \end{equation}
    Hence, it suffices to estimate the upper bound of
    \begin{equation}\label{key_metric}
        \rho(t):=\rho(F_t,G_t).
    \end{equation} 
    Just as we did in Lemmas \ref{lemma_case1} and \ref{lemma_case2}, we divide the estimate into two cases.\\
    \emph{Case 1:} $\Delta(t)< D(\tilde{R})$, then $\rho(t)=f(r_s(t))$. By \eqref{ineq02} and Lemma \ref{lemma_case2} we have
     \begin{equation}\label{distribution_case2}
         \begin{aligned}
        \d \rho(t)
        \leq& -c_2f(r_s(t))\,\d t+\gamma^{-1}\Li f^{\prime}(0)\big(\Wcal_1(\Lcal(X_t),\mu_X^*)+\abs{F_t}\big)\,\d t\\
        &-\frac{\gamma\alpha}{4}f^{\prime}(R_1)\abs{F_t}\d t+(\alpha+\frac{1}{2})\delta\gamma\,\d t +\d \mathbf{M}_t\\
        \leq& -c_2f(r_s(t))\,\d t+\gamma^{-1}\Li f^{\prime}(0)\Wcal_1(\Lcal(X_t),\mu_X^*)\,\d t\\
        &-\frac{\gamma\alpha}{8}f^{\prime}(R_1)\abs{F_t}\d t+(\alpha+\frac{1}{2})\delta\gamma\,\d t +\d \mathbf{M}_t.
        \end{aligned}
     \end{equation}  
     Here we use 
     \[\Li\leq \mathbf{C}_0\frac{L_K+L_g}{4}=f^{\prime}(R_1)f^{\prime}(0)^{-1}\frac{L_K+L_g}{4}=\frac{\gamma^2\alpha}{8}f^{\prime}(R_1)f^{\prime}(0)^{-1}.\]
    \emph{Case 2:} $\Delta(t)\geq D(\tilde{R})$, then applying Ito's formula we get, by $f^{\prime\prime}\leq0$ and Lemma \ref{lemma_case1},
     \[
     \begin{aligned}
        \d &\rho(t)\leq \epsilon_0 f^{\prime}(D(\tilde{R})+\epsilon_0 r_l(t))\Big(-c_1 r_l(t)\\
        &+\gamma^{-1}\Li\big(\Wcal_1(\Lcal(X_t),\mu_X^*)+\abs{F_t}\big)\Big)\d t+\d \tilde{\mathbf{M}}_t,
    \end{aligned}\]
    where $(\tilde{\mathbf{M}}_t)$ is a martingale.
    Note that
    \[
    \begin{aligned}
        -\epsilon_0 &f^{\prime}(D(\tilde{R})+\epsilon_0 r_l(t))c_1 r_l(t)=-\frac{1}{2}\cdot2 f^{\prime}(D(\tilde{R})+\epsilon_0 r_l(t))\epsilon_0 c_1 r_l(t)\\
        &\leq -\frac{1}{2}\Big(\big(\inf\limits_{r\geq0}\frac{rf^{\prime}(r)}{f(r)}\big)\frac{\epsilon_0 c_1 r_l(t)}{D(\tilde{R})+\epsilon_0 r_l(t)}\rho(t)+ f^{\prime}(R_1)\epsilon_0 c_1 r_l(t)\Big)\\
        &\leq -\frac{1}{2}f^{\prime}(R_1)f^{\prime}(0)^{-1}\frac{\epsilon_0 c_1 r_l(t)}{D(\tilde{R})+\epsilon_0 r_l(t)}\rho(t)-\frac{1}{2}f^{\prime}(R_1)\epsilon_0 c_1 r_l(t).
    \end{aligned}\]
    Since $\Delta(t)\geq D(\tilde{R})$, we have (recall the definition of $E_0$ in \eqref{epsilon}) 
    \[\frac{r_l(t)}{D(\tilde{R})+\epsilon_0 r_l(t)}\geq \frac{r_l(t)}{r_s(t)}\geq E_0,\]
    which yields 
    \[
    \begin{aligned}
        -\epsilon_0 f^{\prime}(D(\tilde{R})+\epsilon_0 r_l(t))c_1 r_l(t)\leq -\frac{1}{2}f^{\prime}(R_1)f^{\prime}(0)^{-1}c_1\epsilon_0 E_0 \rho(t)-\frac{1}{2}f^{\prime}(R_1)\epsilon_0 c_1 r_l(t).
    \end{aligned}\]
    Hence, utilizing \eqref{auxiliary_Lipschitz}, $\epsilon_0\leq 1/2$ and $r_l(t)\geq \sqrt{\kappa\gamma^{-2}}\abs{F_t}$ we get 
    \begin{equation}\label{distribution_case1}
        \begin{aligned}
        \d\rho(t)\leq& -\frac{1}{2}f^{\prime}(R_1)f^{\prime}(0)^{-1}c_1\epsilon_0 E_0 \rho(t)\,\d t-\frac{1}{2}f^{\prime}(R_1)\epsilon_0 c_1 r_l(t)\,\d t\\
        &+\epsilon_0 \gamma^{-1}\Li f^{\prime}(0)\big(\Wcal_1(\Lcal(X_t),\mu_X^*)+\abs{F_t}\big)\,\d t+\d \mathbf{\tilde{M}}_t\\
        \leq& -\frac{1}{2}f^{\prime}(R_1)f^{\prime}(0)^{-1}c_1\epsilon_0 E_0 \rho(t)\,\d t-\frac{1}{2}f^{\prime}(R_1)\epsilon_0 c_1 \sqrt{\kappa\gamma^{-2}}\abs{F_t}\d t\\
        &+\frac{1}{2} \gamma^{-1}\Li f^{\prime}(0)\big(\Wcal_1(\Lcal(X_t),\mu_X^*)+\abs{F_t}\big)\,\d t+\d \mathbf{\tilde{M}}_t.
    \end{aligned}
    \end{equation}

    Combining the above two cases and taking expectations on \eqref{distribution_case2} and \eqref{distribution_case1}, we obtain by \eqref{ineq02}
    \begin{equation}\label{ODI}
        \begin{aligned}
         \frac{\d}{\d t}\e[\rho(t)]\leq& -c_3 \e[\rho(t)]-\min\big\{\frac{\gamma\alpha}{8}f^{\prime}(R_1),\frac{1}{2}f^{\prime}(R_1)\epsilon_0 c_1 \sqrt{\kappa\gamma^{-2}}\big\}\e[\abs{F_t}]\\
         &+\frac{1}{2}\gamma^{-1}\Li f^{\prime}(0)\e[\abs{F_t}]+\frac{1}{2}\gamma^{-1}\Li f^{\prime}(0)\Wcal_1(\Lcal(X_t),\mu_X^*)+(\alpha+\frac{1}{2})\delta\\
         \leq& -c_3 \e[\rho(t)]+\frac{1}{2}\gamma^{-1}\Li f^{\prime}(0)\Wcal_1(\Lcal(X_t),\mu_X^*)+(\alpha+\frac{1}{2})\delta,
        \end{aligned}
    \end{equation}
    where 
    \begin{equation}\label{c_3}
        c_3=\min\{c_2,\frac{1}{2}f^{\prime}(R_1)f^{\prime}(0)^{-1}c_1\epsilon_0 E_0\}=f^{\prime}(0)^{-1}\min\{2\gamma^{-1},\frac{\gamma}{4}\kappa^{-1},\frac{1}{2}\kappa^{-1}c_1\epsilon_0 E_0\}.
    \end{equation}
    Here we use 
    \begin{equation}\label{ineq03}
        \gamma^{-1}\Li f^{\prime}(0)\leq \min\big\{\frac{\gamma\alpha}{8}f^{\prime}(R_1),\frac{1}{2}f^{\prime}(R_1)\epsilon_0 c_1 \sqrt{\kappa\gamma^{-2}}\big\},
    \end{equation}
    since it holds $\Li\leq \frac{\gamma^2\alpha}{8}f^{\prime}(R_1)f^{\prime}(0)^{-1}$ and 
    \[
    \begin{aligned}
        \Li&\leq f^{\prime}(R_1)f^{\prime}(0)^{-1}\frac{\gamma\tau}{12}\sqrt{2\kappa}\min\{1,\alpha\}\\
        &=f^{\prime}(R_1)f^{\prime}(0)^{-1}c_1\sqrt{\kappa}\cdot \frac{\sqrt{2}}{6}\min\{1,\alpha\}\\
        &= \frac{1}{2}f^{\prime}(R_1)f^{\prime}(0)^{-1}c_1\sqrt{\kappa}\cdot \frac{\sqrt{2}}{3}\min\{1,\alpha,\sqrt{\alpha}\}\\
        &\leq \frac{1}{2}f^{\prime}(R_1)f^{\prime}(0)^{-1}c_1\sqrt{\kappa}\cdot \frac{1}{2}\min\{1,\alpha,\frac{2}{3}\sqrt{2\alpha}\}\\
        &\leq \frac{1}{2}f^{\prime}(R_1)f^{\prime}(0)^{-1}c_1\sqrt{\kappa}\cdot \frac{1}{2}\min\{1,\alpha,\frac{2\alpha\gamma}{3\sqrt{L_K}}\}\\
        &=\frac{1}{2}f^{\prime}(R_1)f^{\prime}(0)^{-1}\epsilon_0 c_1\sqrt{\kappa}.
    \end{aligned}\]
    By \cite[Theorem 12]{schuh2022global}, we have 
    $$\Wcal_1(\Lcal(X_t),\mu_X^*)\leq \Wcal_1(\Lcal(X_t,V_t),\mu^*)\leq C\exp(-ct)$$ 
    for some constants $c,C>0$. Hence, applying Gronwall's inequality to \eqref{ODI} we have
    \[\e[\rho(t)]\leq C\exp(-ct)+C\delta\] for some $c,C>0$ independent of $\delta$ and $t$.
    Thus we obtain by \eqref{metric_equi}
    \begin{align*}
        \e\Big[\Wcal_1\big(\Ecal_t(X,V),\Ecal_t(Y,U)\big)\Big]\leq& \int_0^1 \e\Big[\abs{\big(\hat{X}_{ts}-\hat{Y}_{ts},\hat{V}_{ts}-\hat{U}_{ts}\big)}\Big]\d s\\
        \leq& \frac{1}{\mathbf{C}_1}\int_0^1 \e[\rho(ts)]\,\d s\\
        \leq &C(\frac{1}{t}+\delta)
    \end{align*}
    for all $\delta>0$ and some constant $C>0$ independent of $t$ and $\delta$. Let $\delta\to0$, we have
    \[\e\Big[\Wcal_1\big(\Ecal_t(X,V),\Ecal_t(Y,U)\big)\Big]\leq \frac{C}{t}.\]

    Combining the above with Theorem \ref{MarkovConvergence}, we complete the proof of Theorem \ref{MeanFieldConvergence}.
    
\subsection{Convergence of self-interacting dynamics}\label{SelfInteracting}
    This section is dedicated to the proof of Theorem \ref{SelfInteractingConvergence}. We follow the notations of the previous section and consider the following reflection coupling:
    \begin{align}
    \label{sde3_coupling}
        &\left\{\begin{aligned}
            \d \hat{Z}_t&=\hat{W}_t\,\d t,\\
            \d \hat{W}_t&=\Big(\be(\hat{Z}_t)+\int_{\Real^d}\bi(\hat{Z}_t,x)\Ecal(\hat{Z}_t)(\d x)-\gamma \hat{W}_t\Big)\d t\\
            &\quad+\sqrt{2\gamma}\mathbf{\lambda}(F_t,G_t)\,\d B_t+\sqrt{2\gamma}\mathbf{\pi}(F_t,G_t)\,\d\hat{B}_t,
        \end{aligned}\right.\\
    \label{sde2.2_coupling}
        &\left\{\begin{aligned}
            \d \hat{Y}_t&=\hat{U}_t\,\d t,\\
            \d \hat{U}_t&=\Big(\be(\hat{Y}_t)+\int_{\Real^d}\bi(\hat{Y}_t,x)\mu_X^*(\d x)-\gamma \hat{U}_t\Big)\d t\\
            &\quad+\sqrt{2\gamma}\mathbf{\lambda}(F_t,G_t)(I_d-2e_t e_t^{\transpose})\,\d B_t+\sqrt{2\gamma}\mathbf{\pi}(F_t,G_t)\,\d \hat{B}_t,
        \end{aligned}\right.
    \end{align}
     where $$Z_0=\hat{Z}_0=Y_0=\hat{Y}_0,\quad W_0=\hat{W}_0=U_0=\hat{U}_0.$$
    After similar computation to \eqref{ODI}, we acquire by \eqref{ineq03}
    \begin{align*}
        \frac{\d}{\d t}\e[\rho(t)]\leq& -c_3 \e[\rho(t)]-\min\big\{\frac{\gamma\alpha}{8}f^{\prime}(R_1),\frac{1}{2}f^{\prime}(R_1)\epsilon_0 c_1 \sqrt{\kappa\gamma^{-2}}\big\}\e[\abs{F_t}]\\
         &+\frac{1}{2}\gamma^{-1}\Li f^{\prime}(0)\e[\abs{F_t}]+\frac{1}{2}\gamma^{-1}\Li f^{\prime}(0)\e[\Wcal_1(\Ecal_t(Z),\mu_X^*)]+(\alpha+\frac{1}{2})\delta\\
         \leq& -c_3 \e[\rho(t)]-\frac{1}{2}\gamma^{-1}\Li f^{\prime}(0)\e[\abs{F_t}]\\
         & +\frac{1}{2}\gamma^{-1}\Li f^{\prime}(0)\e[\Wcal_1(\Ecal_t(Z),\mu_X^*)]+(\alpha+\frac{1}{2})\delta\\
         \leq& -c_3 \e[\rho(t)]-\frac{1}{2}\gamma^{-1}\Li f^{\prime}(0)\e[\abs{F_t}]+\frac{1}{2}\gamma^{-1}\Li f^{\prime}(0)\e[\Wcal_1(\Ecal_t(Z),\Ecal_t(Y))]\\
         &+\frac{1}{2}\gamma^{-1}\Li f^{\prime}(0)\e[\Wcal_1(\Ecal_t(Y),\mu_X^*)]+(\alpha+\frac{1}{2})\delta\\
         \leq& -c_3 \e[\rho(t)]-\frac{1}{2}\gamma^{-1}\Li f^{\prime}(0)\e[\abs{F_t}]+\frac{1}{2}\gamma^{-1}\Li f^{\prime}(0)\int_0^1\e[\abs{F_{ts}}]\d s\\
         &+\frac{1}{2}\gamma^{-1}\Li f^{\prime}(0)\e[\Wcal_1(\Ecal_t(Y),\mu_X^*)]+(\alpha+\frac{1}{2})\delta.
    \end{align*}
    By Theorem \ref{MarkovConvergence}, we have 
    $$\e[\Wcal_1(\Ecal_t(Y),\mu_X^*)]\leq \e[\Wcal_1(\Ecal_t(Y,U),\mu^*)]\leq M (1\wedge t^{-\varepsilon})$$
    for any~$0<\varepsilon<\min\{\frac{1}{4},\frac{1}{2d}\}$. Here $M>0$ is a constant independent of $t$, whose value may vary through lines. We take $\delta=\delta(t)=M(1\wedge t^{-1})$, write $$h(t):=\e[\rho(t)],\quad g(t):=f^{\prime}(0)\e[\abs{F_t}],$$
    and get
    \begin{equation}
        \frac{\d}{\d t}h(t)\leq -c_3 h(t)-\frac{1}{2}\gamma^{-1}\Li g(t)+\frac{1}{2}\gamma^{-1}\Li\int_0^1 g(ts)\,\d s+M(1\wedge t^{-\varepsilon}).
    \end{equation}
    Integrating from $0$ to $t$ and then dividing by $t$ on both sides, we have
    \begin{align*}
        0\leq h(t)/t\leq &-c_3\int_0^1 h(ts)\,\d s-\frac{1}{2}\gamma^{-1}\Li\int_0^1 g(ts)\,\d s\\
        &+\frac{1}{2}\gamma^{-1}\Li\int_0^1\int_0^1 g(ts_1 s_2)\,\d s_1 \d s_2+M(1\wedge t^{-\varepsilon}).
    \end{align*}
    We write 
    \begin{equation}\label{C3}
        \mathbf{C}_3=c_3 \mathbf{C}_1 f^{\prime}(0)^{-1},
    \end{equation}
    where $c_3$ and $\mathbf{C}_1$ are given by \eqref{c_3} and \eqref{C1} respectively.
    For~$0<\varepsilon<\min\{\frac{1}{4},\frac{1}{2d},\frac{\mathbf{C}_3}{\mathbf{C}_3+\gamma^{-1}\Li/2}\}$, we fix $M$ and get
    \begin{equation}\label{continuity_method}        
    c_3 t^{\varepsilon}\int_0^1 h(ts)\,\d s\leq -\frac{1}{2}\gamma^{-1}\Li t^{\varepsilon}\int_0^1 g(ts)\,\d s+\frac{1}{2}\gamma^{-1}\Li t^{\varepsilon}\int_0^1\int_0^1 g(ts_1 s_2)\,\d s_1 \d s_2+M.
    \end{equation}
        
    Now we estimate $h(t)$ and $g(t)$ through \eqref{continuity_method}, taking advantage of continuity. We claim that 
    \begin{equation}\label{neq_claim}
        \sup\limits_{t\geq0} t^{\varepsilon}\int_0^1 g(ts)\,\d s\leq \frac{2}{\mathbf{C}_3-\frac{1}{2}\gamma^{-1}\Li\frac{\varepsilon}{1-\varepsilon}}M.
    \end{equation}
    Otherwise, we assume 
    \[T:=\inf\Big\{t\geq0:\,t^{\varepsilon}\int_0^1 g(ts)\,\d s> \frac{2}{\mathbf{C}_3-\frac{1}{2}\gamma^{-1}\Li\frac{\varepsilon}{1-\varepsilon}}M\Big\},\]
    and instantly get
    \[\frac{2}{\mathbf{C}_3-\frac{1}{2}\gamma^{-1}\Li\frac{\varepsilon}{1-\varepsilon}}M=T^{\varepsilon}\int_0^1 g(Ts)\,\d s=\sup\limits_{0\leq t\leq T} t^{\varepsilon}\int_0^1 g(ts)\,\d s.\]
    By \eqref{metric_equi}, we have
    \begin{align*}
        &\frac{2\mathbf{C}_3}{\mathbf{C}_3-\frac{1}{2}\gamma^{-1}\Li\frac{\varepsilon}{1-\varepsilon}}M\leq c_3 T^{\varepsilon}\int_0^1 h(Ts)\,\d s\\
        &\leq -\frac{1}{2}\gamma^{-1}\Li T^{\varepsilon}\int_0^1 g(Ts)\,\d s+\frac{1}{2}\gamma^{-1}\Li\int_0^1 s_1^{-\varepsilon}(Ts_1)^{\varepsilon}\int_0^1 g(Ts_1 s_2)\,\d s_2 \d s_1+M\\
        &\leq -\frac{1}{2}\gamma^{-1}\Li T^{\varepsilon}\int_0^1 g(Ts)\,\d s+\frac{1}{2}\gamma^{-1}\Li\int_0^1 s_1^{-\varepsilon}\,\d s_1 T^{\varepsilon}\int_0^1 g(T s_2)\,\d s_2 +M\\
        &=\frac{1}{2}\gamma^{-1}\Li \frac{\varepsilon}{1-\varepsilon}\frac{2}{\mathbf{C}_3-\frac{1}{2}\gamma^{-1}\Li\frac{\varepsilon}{1-\varepsilon}}M+M\\
        &=\frac{\mathbf{C}_3+\frac{1}{2}\gamma^{-1}\Li\frac{\varepsilon}{1-\varepsilon}}{\mathbf{C}_3-\frac{1}{2}\gamma^{-1}\Li\frac{\varepsilon}{1-\varepsilon}}M\\
        &<\frac{2\mathbf{C}_3}{\mathbf{C}_3-\frac{1}{2}\gamma^{-1}\Li\frac{\varepsilon}{1-\varepsilon}}M,
    \end{align*}
    which leads to contradiction! Hence, \eqref{neq_claim} holds, and we obtain
    \[\frac{\d}{\d t}h(t)\leq -c_3 h(t)+C(1\wedge t^{-\varepsilon})\]
    for some $C>0$. By Gronwall's inequality, we have 
    $h(t)\leq C t^{-\varepsilon}$ and thus by \eqref{metric_equi}
    \[\e[\Wcal_1(\Ecal_t(Z,W),\Ecal(Y,U))]\leq \frac{1}{\mathbf{C}_1}\int_0^1 h(ts)\,\d s\leq C t^{-\varepsilon}\]
    for some constant $C>0$.
    Combining the above with Theorem \ref{MarkovConvergence}, we complete the proof of Theorem \ref{SelfInteractingConvergence}.
    \begin{re}
        If $R=0$, \eqref{ineq02} can be reduced to
        \begin{equation}\label{ineq_special}
            \Li\leq \frac{\tau\gamma}{8}\sqrt{\kappa},
        \end{equation}
        and empirical ergodicity of \eqref{sde1} and \eqref{sde3} still holds. In fact, we only need Lemma~\ref{lemma_case1} for proofs of Theorems \ref{MeanFieldConvergence} and \ref{SelfInteractingConvergence} in the $R=0$ scenario.
    \end{re}

\section{Proof of Theorem \ref{MarkovConvergence}}\label{appendix}
    By \cite[Theorem 5.1]{cao2023empirical}, empirical ergodicity of the Markovian equation \eqref{sde2}, \ie~Theorem~\ref{MarkovConvergence}, follows from its moment control and contractivity. 
    Hence, it suffices to prove the following two propositions.
    \begin{prop}\label{L2}
        Suppose Assumption \ref{dissi}, \eqref{ineq01} and \eqref{ineq02} hold. Then for any solution $(Y_t,U_t)$ of \eqref{sde2} whose initial value satisfies 
        \[\e[\abs{Y_0}^2+\abs{U_0}^2]<\infty,\]
        there exist a constant $$M_2=M_2\Big(d,\gamma,L_K,L_g,\kappa,\e[\abs{Y_0}^2+\abs{U_0}^2]\Big)\in(0,\infty)$$
        such that 
        \[\sup\limits_{t\geq0}\e[\abs{Y_t}^2+\abs{U_t}^2]\leq M_2.\]
    \end{prop}
    \begin{prop}\label{MarkovContraction_prop}
        Suppose Assumption \ref{dissi}, \eqref{ineq01} and \eqref{ineq02} hold.
         Then there exist constants $M_1>0$ and $c>0$ such that for any solutions $(Y_t,U_t)$ and $(\hat{Y}_t,\hat{U}_t)$ of \eqref{sde2}, whose initial values have finite second moment, we have 
        \[W_1\big(\Lcal(Y_t,U_t),\Lcal(\hat{Y}_t,\hat{U}_t)\big)\leq M_1 \exp(-ct)W_1\big(\Lcal(Y_0,U_0),\Lcal(\hat{Y}_0,\hat{U}_0)\big).\]
    \end{prop}
    Note that Propositions \ref{L2} and \ref{MarkovContraction_prop} also yield the existence and uniqueness of the invariant measure $\mu^*$ of \eqref{sde2}, which is identical to that of \eqref{sde1} (cf.~\cite{wang2018distribution}). 
\subsection{Proof of Proposition \ref{L2}}
    For the moment control of the solution $(Y_t,U_t)$ of \eqref{sde2}, we consider 
    \[r_l(Y_t,U_t)^2=\gamma^{-2}\langle KY_t,Y_t\rangle+\frac{1}{2}\abs{(1-2\tau)Y_t+\gamma^{-1}U_t}^2+\frac{1}{2}\gamma^{-2}\abs{U_t}^2.\]
    Following the ideas of~\cite[Lemma 23]{schuh2022global}, we apply Ito's formula to $r_l(Y_t,U_t)^2$ and get
    \begin{align*}
        \d &r_l(Y_t,U_t)^2\\
        \leq\,& 2\gamma^{-2}\langle KY_t,U_t\rangle\,\d t+\big((1-2\tau)^2\langle Y_t,U_t\rangle+\gamma^{-1}(1-2\tau)\abs{U_t}^2\big)\,\d t\\
        &-\gamma^{-1}(1-2\tau)\big(\langle KY_t,Y_t\rangle+\gamma\langle Y_t,U_t\rangle\big)\,\d t\\
        &+2\gamma^{-2}\big(-\langle KY_t,U_t\rangle+L_g\abs{Y_t}\abs{U_t}-\gamma\abs{U_t}^2\big)\,\d t\\ 
        &+\gamma^{-1}\abs{(1-2\tau)Y_t+\gamma^{-1}U_t}\big(\Li\abs{Y_t}+\Li\norm{\mu^*}_1+\big\vert\bi(0,0)\big\vert\big)\,\d t\\
        &+\gamma^{-1}(1-2\tau)\big(L_g\abs{Y_t}^2+\abs{g(0)}\abs{Y_t}\big)\cdot\mathbb{I}_{\{\abs{Y_t}<R\}}\,\d t
        +2\gamma^{-2}\abs{U_t}\abs{g(0)}\d t+2\gamma^{-1}d\,\d t\\
        &+\sqrt{2\gamma^{-1}}\big((1-2\tau)Y_t+2\gamma^{-1}U_t\big)^{\transpose}\d B_t\\
        \leq& -\gamma^{-1}(1-2\tau)\langle KY_t,Y_t\rangle\,\d t-2\tau\gamma\big(\gamma^{-2}\abs{U_t}^2+\gamma^{-1}(1-2\tau)\langle Y_t,U_t\rangle\big)\,\d t\\
        &+\gamma^{-3}L_g^2\abs{Y_t}^2\d t+\gamma^{-1}\abs{(1-2\tau)Y_t+\gamma^{-1}U_t}\big(\Li\abs{Y_t}+\Li\norm{\mu^*}_1+\big\vert\bi(0,0)\big\vert\big)\,\d t\\
        &+\gamma^{-1}(1-2\tau)\big(L_g\abs{Y_t}^2+\abs{g(0)}\abs{Y_t}\big)\cdot\mathbb{I}_{\{\abs{Y_t}<R\}}\,\d t
        +2\gamma^{-2}\abs{U_t}\abs{g(0)}\d t+2\gamma^{-1}d\,\d t\\
        &+\sqrt{2\gamma^{-1}}\big((1-2\tau)Y_t+2\gamma^{-1}U_t\big)^{\transpose}\d B_t.
    \end{align*}
    Taking expectation, we acquire
    \begin{align*}
        \frac{\d}{\d t} \e[r_l(Y_t,U_t)^2]\leq& -\gamma^{-1}(1-2\tau)\e[\langle KY_t,Y_t\rangle]+\gamma^{-3}L_g^2\e[\abs{Y_t}^2]\\
        &-2\tau\gamma\big(\gamma^{-2}\e[\abs{U_t}^2]+\gamma^{-1}(1-2\tau)\e[\langle Y_t,U_t\rangle]\big)\\
        &+\gamma^{-1}\e\Big[\abs{(1-2\tau)Y_t+\gamma^{-1}U_t}\big(\Li\abs{Y_t}+\Li\norm{\mu^*}_1+\big\vert\bi(0,0)\big\vert\big)\Big]\\
        &+\gamma^{-1}(1-2\tau)(L_g R^2+\abs{g(0)}R)
        +2\gamma^{-2}\abs{g(0)}\e[\abs{U_t}]+2\gamma^{-1}d.
    \end{align*}
    By \eqref{ineq02} and Young's inequality, we have
    \begin{align*}
        \gamma^{-1}&\e\Big[\abs{(1-2\tau)Y_t+\gamma^{-1}U_t}\big(\Li\abs{Y_t}+\Li\norm{\mu^*}_1+\big\vert\bi(0,0)\big\vert\big)\Big]\\
        \leq& \frac{\tau\sqrt{\kappa}}{8}\e\big[\abs{(1-2\tau)Y_t+\gamma^{-1}U_t}\abs{Y_t}\big]\\
        &+\gamma^{-1}\e\big[\abs{(1-2\tau)Y_t+\gamma^{-1}U_t}\big]\big(\big\vert\bi(0,0)\big\vert+\Li\norm{\mu^*}_1\big)\\
        \leq& \frac{\tau\gamma}{4}\Big(\frac{1}{4}\e\big[\abs{(1-2\tau)Y_t+\gamma^{-1}U_t}^2\big]+\frac{\kappa\gamma^{-2}}{4}\e[\abs{Y_t}^2]\Big)\\
        &+\frac{\tau\gamma}{16}\e\big[\abs{(1-2\tau)Y_t+\gamma^{-1}U_t}^2\big]+\frac{4}{\tau\gamma^3}\big(\big\vert\bi(0,0)\big\vert+\Li\norm{\mu^*}_1\big)^2\\
        \leq& \frac{\tau\gamma}{4}\Big(\kappa\gamma^{-2}\e[\abs{Y_t}^2]+\frac{1}{2}\e\big[\abs{(1-2\tau)Y_t+\gamma^{-1}U_t}^2\big]\Big)\\
        &+\frac{4}{\tau\gamma^3}\big(\big\vert\bi(0,0)\big\vert+\Li\norm{\mu^*}_1\big)^2,
    \end{align*}
    and
    \[2\gamma^{-2}\abs{g(0)}\e[\abs{U_t}]\leq \frac{\tau\gamma}{4}\e[\gamma^{-2}\abs{U_t}^2]+\frac{4}{\tau\gamma^3}\abs{g(0)}^2.\]
    Also, by \eqref{tau} we have
    \begin{align*}
        \langle x,\,(-(1-4\tau)\gamma^{-1}K+\gamma^{-3}L_g^2 I_d)x\rangle &\leq (-\frac{1}{2}\kappa\gamma^{-1}+\gamma^{-3}L_g^2)\abs{x}^2\\
        &\leq -\tau\gamma\abs{x}^2\\
        &\leq -\tau\gamma(1-2\tau)^2\abs{x}^2
    \end{align*}
    for all $x\in\Real^d$.
    Hence we obtain
    \begin{align*}
        \frac{\d}{\d t} &\e[r_l(Y_t,U_t)^2]\\
        \leq& -2\tau\gamma\e\big[\gamma^{-2}\langle KY_t,Y_t\rangle+\frac{1}{2}\abs{(1-2\tau)Y_t+\gamma^{-1}U_t}^2+\frac{1}{2}\gamma^{-2}\abs{U_t}^2\big]\\
        &+\tau\gamma\Big(\kappa\gamma^{-2}\e[\abs{Y_t}^2]+\frac{1}{2}\e\big[\abs{(1-2\tau)Y_t+\gamma^{-1}U_t}^2\big]+\frac{1}{2}\e[\gamma^{-2}\abs{U_t}^2]\Big)\\
        &+\frac{4}{\tau\gamma^3}\big(\big\vert\bi(0,0)\big\vert+\Li\norm{\mu^*}_1\big)^2+\frac{4}{\tau\gamma^3}\abs{g(0)}^2\\
        &+\gamma^{-1}(1-2\tau)(L_g R^2+\abs{g(0)}R)+2\gamma^{-1}d\\
        \leq& -\tau\gamma\e[r_l(Y_t,U_t)^2]+\frac{4}{\tau\gamma^3}\big(\big\vert\bi(0,0)\big\vert+\Li\norm{\mu^*}_1\big)^2\\
        &+\frac{4}{\tau\gamma^3}\abs{g(0)}^2+\gamma^{-1}(1-2\tau)(L_g R^2+\abs{g(0)}R)+2\gamma^{-1}d.
    \end{align*}
    By Gronwall's inequality, there exists a constant $C>0$ such that
    \[\e[r_l(Y_t,U_t)^2]\leq C\]
    for all $t\geq0$. We note that $r_l$ induces a norm on $\Real^{2d}$ equivalent to the Euclidean norm, and thus complete the proof of Proposition \ref{L2}.

\subsection{Proof of Proposition \ref{MarkovContraction_prop}}
    We consider the following reflection coupling: 
    \begin{align}\label{markov1_coupling}
        &\left\{\begin{aligned}
            \d \hat{Y}_t^1&=\hat{U}_t^1\d t,\\
            \d \hat{U}_t^1&=\Big(\be(\hat{Y}_t^1)+\int_{\Real^d}\bi(\hat{Y}_t^1,x)\mu_X^*(\d x)-\gamma \hat{U}_t^1\Big)\d t\\
            &\quad+\sqrt{2\gamma}\mathbf{\lambda}(F_t,G_t)\,\d B_t+\sqrt{2\gamma}\mathbf{\pi}(F_t,G_t)\,\d \hat{B}_t.
        \end{aligned}\right.\\
        \label{markov2_coupling}
        &\left\{\begin{aligned}
            \d \hat{Y}_t^2&=\hat{U}_t^2\d t,\\
            \d \hat{U}_t^2&=\Big(\be(\hat{Y}_t^2)+\int_{\Real^d}\bi(\hat{Y}_t^2,x)\mu_X^*(\d x)-\gamma \hat{U}_t^2\Big)\d t\\
            &\quad+\sqrt{2\gamma}\mathbf{\lambda}(F_t,G_t)(I_d-2e_t e_t^{\transpose})\,\d B_t+\sqrt{2\gamma}\mathbf{\pi}(F_t,G_t)\,\d \hat{B}_t.
        \end{aligned}\right.
    \end{align}
    Here we follow the notations of the Sections \ref{meanfield} and \ref{SelfInteracting}, \ie 
    \[F_t=\hat{Y}_t^1-\hat{Y}_t^2,\quad G_t=\hat{U}_t^1-\hat{U}_t^2,\quad H_t=F_t+\gamma^{-1}G_t,\quad e_t=\frac{H_t}{\abs{H_t}},\]
    and (recall \eqref{large_small_norms}, \eqref{Delta} and \eqref{key_metric})
    \[r_l(t)=r_l(F_t,G_t),\quad r_s(t)=r_s(F_t,G_t),\quad \Delta(t)=\Delta(F_t,G_t),\quad \rho(t)=\rho(F_t,G_t).\]

    Analogously to \eqref{ODI}, we acquire
    \begin{align*}
        \frac{\d}{\d t}\e[\rho(t)]\leq &-c_3 \e[\rho(t)]-\min\big\{\frac{\gamma\alpha}{8}f^{\prime}(R_1),\frac{1}{2}f^{\prime}(R_1)\epsilon_0 c_1 \sqrt{\kappa\gamma^{-2}}\big\}\e[\abs{F_t}]\\
         &+\frac{1}{2}\gamma^{-1}\Li f^{\prime}(0)\e[\abs{F_t}]+(\alpha+\frac{1}{2})\delta\\
         \leq& -c_3 \e[\rho(t)]+(\alpha+\frac{1}{2})\delta,
    \end{align*}
    where $c_3$ is given by \eqref{c_3}.
    By Gronwall's inequality, we have 
    \[\e[\rho(t)]\leq \e[\rho(0)]\exp(-c_3 t)+C\delta,\]
    and thus
    \[\begin{aligned}
    \Wcal_1(\Lcal(Y_t^1,U_t^1),\Lcal(Y_t^2,U_t^2))&=\Wcal_1(\Lcal(\hat{Y}_t^1,\hat{U}_t^1),\Lcal(\hat{Y}_t^2,\hat{U}_t^2))\\
    &\leq \frac{1}{\mathbf{C}_1}\e[\rho(t)]\\
    &\leq C\exp(-c_3 t)\Wcal_1(\Lcal(Y_0^1,U_0^1),\Lcal(Y_0^2,U_0^2))+C\delta,
    \end{aligned}\]
    where $C>0$ represents a constant independent of $t$ and $\delta$.
    Let $\delta\to 0$, and we finish the proof of Proposition \ref{MarkovContraction_prop}.
    
\section{Numerical experiments}\label{experiment}
    We consider the following example of \eqref{sde1} on $\Real^1\times\Real^1$:
    \begin{equation}\label{sde_numerical}
    \left
    \{\begin{aligned}
        \d X_t&=V_t\,\d t,\\
        \d V_t&=\big(-2X_t+k\e[X_t]-3V_t\big)\,\d t+\sqrt{6}\,\d B_t,
    \end{aligned}
    \right.
    \end{equation}
    with $\be(x)=-2x=- \nabla x^2$, $\bi(x,x^{\prime})=kx^{\prime}$, and $\gamma=3$. Here $k\geq0$ is a constant. By symmetry, the invariant measure $\mu^*$ of \eqref{sde_numerical} is identical to that of the classical Langevin dynamics
    \begin{equation}\label{markov_numerical}
        \left
    \{\begin{aligned}
        \d Y_t&=U_t\,\d t,\\
        \d U_t&=\big(-2Y_t-3U_t\big)\,\d t+\sqrt{6}\,\d B_t,
    \end{aligned}
    \right.
    \end{equation}   
    where we acquire, by solving a Fokker--Planck equation, (cf. \cite{pavliotis2014diffusion})
    \[\mu^*(\d x\d v)=\frac{1}{\sqrt{2}\pi}\exp(-x^2-\frac{1}{2}v^2)\,\d x\d v.\]

    In our experiments, we set the step length $\Delta t=1$, and consider the weighted path-dependent dynamics $(\hat{Z}_t,\hat{W}_t)$ governed by the following equation for all $j\in\ZZ_{\geq 0},\, t\in(j,j+1]$:
    \begin{equation}\label{path_numerical}
        \left
    \{\begin{aligned}
        \d \hat{Z}_t&=\hat{W}_t\,\d t,\\
        \d \hat{W}_t&=\big(-2\hat{Z}_t-3\hat{W}_t+k m_j\big)\,\d t+\sqrt{6}\,\d B_t.
    \end{aligned}
    \right.
    \end{equation} 
    Here the self-interacting term is characterized by the mean value $m_j=\frac{1}{j+1}\sum\limits_{i=0}^{j} \hat{Z}_i$, updated by
    \[m_{j+1}=\frac{j+1}{j+2}m_j+\frac{1}{j+2}\hat{Z}_{j+1}\]
    to save storage space.
    We notice that on each integer time interval $(j,j+1]$, the solution $(\hat{Z}_t,\hat{W}_t)$ satisfies
    \begin{equation}
        \left
    \{\begin{aligned}
        2\hat{Z}_t+\hat{W}_t&=e^{-(t-j)}(2\hat{Z}_j+\hat{W}_j)+km_j\int_j^t e^{-(t-s)}\d s+\sqrt{6}\int_j^t e^{-(t-s)}\d B_s,\\
        \hat{Z}_t+\hat{W}_t&=e^{-2(t-j)}(\hat{Z}_j+\hat{W}_j)+km_j\int_j^t e^{-2(t-s)}\d s+\sqrt{6}\int_j^t e^{-2(t-s)}\d B_s.
    \end{aligned}
    \right.
    \end{equation}
    Thus we obtain the recursive formula  
    \begin{equation}\label{recursive}
        \left
    \{\begin{aligned}
        \hat{Z}_{j+1}&=(2e^{-1}-e^{-2})\hat{Z}_j+(e^{-1}-e^{-2})\hat{W}_j\\
        &\quad+km_j
        (\frac{1}{2}-e^{-1}+\frac{1}{2}e^{-2})+\big(\sqrt{3}\sqrt{1-e^{-2}}-\frac{\sqrt{6}}{2}\sqrt{1-e^{-4}}\big)\xi_j,\\
        \hat{W}_{j+1}&=(-2e^{-1}+2e^{-2})\hat{Z}_j+(-e^{-1}+2e^{-2})\hat{W}_j\\
        &\quad+km_j(e^{-1}-e^{-2})+\big(-\sqrt{3}\sqrt{1-e^{-2}}+\sqrt{6}\sqrt{1-e^{-4}}\big)\xi_j,
    \end{aligned}
    \right.
    \end{equation}
    which enables us to compute $(\hat{Z}_t,\hat{W}_t)$ and compare $\Ecal_t(\hat{Z},\hat{W})$ with $\Ecal_t(Y,U)$. The error is characterized by the value of $\abs{m_j}$. After $8\times10^5$ steps of iteration on $50$ independent sample paths for cases $k=0.4,\,0.8,\,1.2,\,1.6,\,2.0$, we acquire the following graph.
    \begin{figure}[H]
		\centering
		\includegraphics[width=0.85\textwidth]{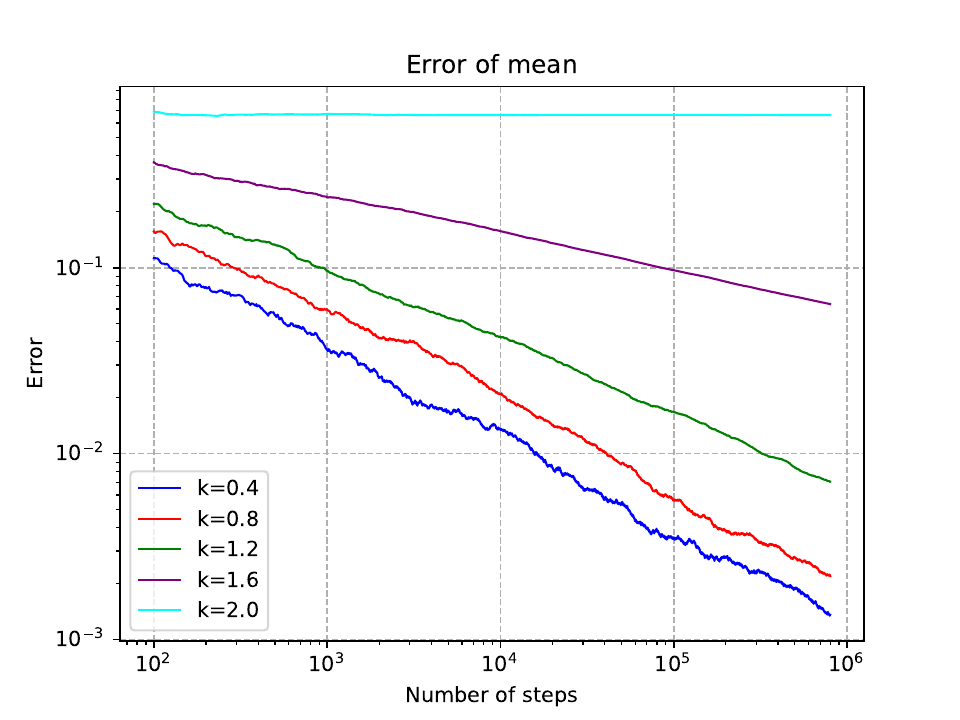}
		\caption{Error of empirical measures of the self-interacting dynamics $(\hat{Z}_t,\hat{W}_t)$.}
    \end{figure}

    Note that none of the above cases satisfy \eqref{ineq02} or \eqref{ineq_special}, which requires $k\leq \frac{\sqrt{2}}{24}$. Judging from the figure, however, empirical measures of $(\hat{Z}_t,\hat{W}_t)$ in all $k\leq 1.6$ scenarios show convergence behavior. The convergence rate decreases as $k$ increases, and divergence occurs at $k=2.0$. These phenomena demonstrate the effectiveness of our algorithm as well as suggest that the theoretical results are by no means optimal, where the weak interaction conditions \eqref{ineq02} and \eqref{ineq_special} can still be relaxed, and Theorem \ref{SelfInteractingConvergence} may be extended to weighted empirical measures as in \cite{du2023empirical}.
\bibliographystyle{plain}
\bibliography{myref.bib}
\end{document}